\documentclass[preprint]{elsarticle}
\usepackage[left=1.5cm,
    right=1.5cm,
    top=2cm,
    bottom=2cm
]{geometry}
\usepackage{mathrsfs}
\usepackage{float}
\usepackage[caption = false]{subfig}
\makeatletter
\newcommand{\setword}[2]{%
  \phantomsection
  #1\def\@currentlabel{\unexpanded{#1}}\label{#2}%
}
\makeatother

\newtheorem{definition}{Definition}[section]
\newtheorem{example}{Example}[section]
\newtheorem{theorem}{Theorem}[section]
\newtheorem{proposition}{Proposition}[section]
\newtheorem{assum}{Assumption}[section]
\newtheorem{note}{Note}[section]
\newproof{proof}{Proof}
\newtheorem{lemma}{Lemma}[section]

\newtheorem{remark}{Remark}[section]

\usepackage{xcolor}
\usepackage{float}
\usepackage{microtype}
\usepackage{times}
\usepackage{graphicx,epsfig}
\usepackage{amsmath}
\usepackage{amssymb}
\usepackage{enumerate}
\usepackage{mathtools}
\everymath{\displaystyle}
\usepackage[linkcolor=blue, urlcolor=blue, citecolor=blue,
colorlinks, bookmarks]{hyperref}
\date{}

\numberwithin{equation}{section}
\journal{}

\begin{document}

\begin{frontmatter}

\title {\textbf{\Large On the horizon problem over the product of self-similar fractals}}

\author[1]{Gurubachan}
\ead{gurubachanmakkar@gmail.com}
\author[2]{S. Verma}
\ead{saurabhverma@iiita.ac.in}
\author[1]{V.V.M.S. Chandramouli\corref{mycorrespondingauthor}}
\cortext[mycorrespondingauthor]{Corresponding author}
\ead{chsarma@iitj.ac.in}

\address[1]{Department of Mathematics\\
Indian Institute of Technology Jodhpur\\
Jodhpur 342030, India.}
\address[2]{Department of Applied Sciences\\
Indian Institute of Information Technology Allahabad\\
Prayagraj 211015, India.}

\begin{abstract}

In this article, we first investigate the box dimension results for the graphs of continuous functions defined on $\mathfrak{F\times F}$, where $\mathfrak{F}$ is the self-similar fractal generated by an iterated function system (IFS) exhibiting the open set condition (OSC). Following this, we manifest the existence of a prevalent surface, defined by a continuous function on $\mathfrak{F\times F}$, that satisfies the horizon property in relation to the box dimension, i.e., the box dimension of the prevalent surface is greater than that of its horizon up to a constant $\eta$, where $\eta$ is the box dimension of the self-similar fractal $\mathfrak{F}$. Finally, we provide an example consolidating that the OSC on the underlying IFS of $\mathfrak{F}$ cannot be dropped.\\



 
\end{abstract}

\end{frontmatter}
\allowdisplaybreaks
\section{Introduction}\label{sec:1}
Fractal geometry provides an effective framework to study irregular and complex structures such as cells of the body, mountains, clouds, coastlines, etc., and consequently it has become an emerging area of research. Indeed, most of the fractals can be characterized as the attractors of some iterated function system (IFS). An IFS $\mathfrak{J}=\left\{\mathbb{R}^n;~g_1,g_2,\dots,g_N\right\}$ is the collection of the finite number of the contraction maps defined on $\mathbb{R}^n$, where $n\in \mathbb{N}$ and $N(\ge2)\in \mathbb{N}$. Hutchinson \cite{JH} was the first who gave a significant tool in the form of an IFS, which is employed to construct the attractors. Corresponding to the IFS $\mathfrak{J}$, it has been precisely demonstrated that there exists a unique non-empty compact subset $\mathfrak{F}$ of $\mathbb{R}^n$ that satisfies the self-referential equation
$$\mathfrak{F}=\bigcup_{i=1}^N g_i(\mathfrak{F}).$$
Now, a natural question arises: how to measure the geometric complexity of the fractals? This was affirmatively answered by a prominent tool that is known as \textit{fractal dimension}, which is used to precisely identify their fractal features. Fractal dimension has various forms, such as the similarity dimension, Hausdorff dimension, box dimension, Assouad dimension, quantization dimension, and intermediate dimension; see \cite{BSS,Fal,Fal2,KZ,FFK2}. Particularly, the estimation of fractal dimension via the mass distribution principle and potential theoretic methods can be seen in \cite{Fal,Fal2}. Whereas in \cite{NPV}, a rigorous technique of computing the Hausdorff dimension has been given that is embedded in the theory of positive linear operators, commonly referred to as the \textit{Perron-Frobenius operators} or \textit{Frobenius-Ruelle operators}, and their spectral radius. Following this technique, an algorithm has been given in \cite{P} to estimate the lower bound of the Hausdorff dimension of a set $\beth$, where $\beth$ is the collection of all complex continued fractions that are of the form $\frac{1}{z_1+}\frac{1}{z_2+}\dots$ with $z_k=u_k+iv_k$, $u_k\in\mathbb{N}$, $v_k\in \mathbb{Z}$, and $k\in\mathbb{N}$. As the graphs of fractal functions are the attractors of some IFS, this leads to computing their dimension and working on some further dimensional approximations and dimensional analysis; see, for instance, \cite{BV,AV2,CA2,DV2,MJ,JV, YLC} and references therein. If we consider a probability vector $\mathfrak{q}=(q_1,q_2,\dots,q_N)$ with the above-defined IFS $\mathfrak{J}$, then by Hutchinson \cite{JH}, we can get a unique Borel probability measure $\lambda$ supported on $\mathfrak{F}$ with
 $$\lambda =\sum_{i=1}^N q_i \lambda\circ g_i^{-1}.$$
 This measure $\lambda$ is known as the fractal measure corresponding to weighted IFS $\mathfrak{J}$. Some recent works regarding calculating the dimension of the fractal measure can be seen in \cite{BV2,AV1,DV1,LSV}.
\par
An interesting problem associated with the fractal dimension that has received a great amount of attention in the last few decades is \textit{the horizon problem for random surfaces}, which was introduced by Falconer \cite{F}. Potential theoretic techniques were employed to determine the bounds for the Hausdorff dimension for the horizon of index-$\alpha$ Brownian fields in \cite{F,FV}. In particular, it has been shown that almost surely all the index-$\frac{1}{2}$ Brownian surfaces satisfy the horizon property for the Hausdorff dimension. Following this, Falconer and Fraser \cite{FF} have shown the existence of a prevalent surface, corresponding to a real-valued continuous function $g$ on $[0,1]^2$, that satisfies the horizon property for the box dimension, i.e.,
$$\dim_B\big(\mathfrak{Gr}(H(g))\big)=\dim_B\big(\mathfrak{Gr}(g)\big)-1,$$
where the horizon $H(g)$ is a real-valued continuous function on $[0,1]$ which is defined by $H(g)(x)=\sup_{y\in[0,1]}g(x,y),$ $\dim_B$ symbolizes the box dimension and $\mathfrak{Gr}(H(g)),$ $ \mathfrak{Gr}(g)$ denote the graphs of $H(g),g$, respectively.  

\par
Motivated by the work mentioned above, we start by defining the horizon property for the box dimension on $\mathfrak{F\times F}$, where $\mathfrak{F}$ is the self-similar fractal and its associated IFS satisfies the open set condition (OSC). We manifest the existence of the prevalent surface in a completely topological metrizable space (defined on $\mathfrak{F}\times\mathfrak{F}$) that satisfies the horizon property in terms of the box dimension. In a precise manner, our primary result is as follows:
\begin{theorem}\label{thm12}
    Consider $\mathfrak{F}$ is a self-similar fractal from $\mathbb{R}^n$ with the box dimension $\eta$ and its corresponding IFS satisfy the OSC. Then,
    \begin{enumerate}
        \item \label{thm12_it2} for each $\xi \in [2\eta,2\eta+1)$, the subset $\mathfrak{A}_{\xi}(\mathfrak{F} \times \mathfrak{F})$ is a prevalent of $\big(\mathcal{D}_{\xi}(\mathfrak{F} \times \mathfrak{F}),\Theta_\xi\big);$
        \item the set $\mathfrak{A}_{2\eta+1}(\mathfrak{F} \times \mathfrak{F})$ is a prevalent subset of $\big(\mathcal{C}(\mathfrak{F} \times \mathfrak{F}),\Theta_\infty\big)$.
    \end{enumerate}
\end{theorem}
We also discuss the existence of a prevalent surface in the H\"older space defined on $\mathfrak{F\times F}$. We also give an example justifying the necessity of OSC for the entire analysis. To consolidate this theory, we show some canonical cases over the compact interval $[0,1]$ and the Sierpi\'nski gasket ($SG$), which are well-known fractals of $\mathbb{R}$ and $\mathbb{R}^2$, respectively.
\par 
The article is structured as follows: In Section \ref{sec:2}, we recall some essential definitions and remarkable notes from \cite{JH,Fal,FF}. In Section \ref{sec:3}, we discuss the main results of this article. Particularly, we define and investigate the horizon property w.r.t. the box dimension for the prevalent surface on $\mathfrak{F\times F}$. In Section \ref{sec:4}, we demonstrate some special cases of the analytical results done in Section \ref{sec:3} in the form of illustrations for some recognized fractals in $\mathbb{R}$ and $\mathbb{R}^2$.

\section{Preliminaries}\label{sec:2}
First, we recall the definitions related to the Hausdorff dimension and the box dimension.

For a fixed $n\in \mathbb{N}$, the collection of all non-empty compact subsets of $\mathbb{R}^n$ is denoted by $\mathscr{K}(\mathbb{R}^n)$ and the Hausdorff distance between any two sets $E_1,E_2$ of $\mathscr{K}(\mathbb{R}^n)$ is defined as
\begin{eqnarray*}
    \mathfrak{H}(E_1,E_2)=\max\Big\{\sup_{a\in E_1}\inf_{b\in E_2}\|a-b\|_2, \sup_{b\in E_2}\inf_{a\in E_1}\|a-b\|_2\Big\},
\end{eqnarray*}
where $\|.\|_2$ stands for the Euclidean norm in $\mathbb{R}^n$.

\begin{definition}\cite{Fal}
Let $\delta>0$ and $s$ be any non-negative number. For any subset $E$ of $\mathbb{R}^n$, its $s$-dimensional Hausdorff (outer) measure is defined as
$$\mathscr{H}^s(E)=\lim_{\delta\to 0^{+}}\mathscr{H}_{\delta}^s(E),$$
where $\mathscr{H}_{\delta}^s(E)=\inf\left\{\sum_{i=1}^{\infty}|\mathcal{V}_i|^s:~ E\subseteq\bigcup_{i=1}^{\infty}\mathcal{V}_i\text{ and } 0< |\mathcal{V}_i|\le \delta\right\}$ and $|\mathcal{V}_i|=\sup_{x,y\in \mathcal{V}_i}\|x-y\|_2$. The Hausdorff dimension of $E$ is defined as
\begin{eqnarray*}
    \dim_{\mathscr{H}}(E)&=& \inf\{s\geq0:~\mathscr{H}^s(E)=0\}\\
    &=&\sup\{s\geq0:~\mathscr{H}^s(E)=\infty\}.
\end{eqnarray*}
\end{definition}

\begin{definition} \cite{Fal}
Let $E$ be any non-empty, bounded subset of $ \mathbb{R}^n$. Then, the respective lower and the upper box dimensions of $E$ are defined as
\begin{eqnarray*}
    \underline{\dim}_B(E)=\varliminf_{\delta\to 0^+}\frac{\log(N_{\delta}(E))}{-\log(\delta)} \quad\text{and}\quad
    \overline{\dim}_B(E)=\varlimsup_{\delta\to 0^+}\frac{\log(N_{\delta}(E))}{-\log(\delta)},
\end{eqnarray*}
 where $N_{\delta}(E)$ denotes the number of $n$-dimensional cubes of $\delta$-mesh that intersect with $E$. 
If $\underline{\dim}_B(E)= \overline{\dim}_B(E)$, then we say that the box dimension of $E$ exists and denote it by $\dim_B(E)~\big(=\underline{\dim}_B(E)= \overline{\dim}_B(E)\big)$.
\end{definition}
Let us recall some definitions from \cite{FF}.
\begin{definition}
    A subset $G$ of a completely metrizable topological vector space $Y$ is said to be prevalent if 
    \begin{enumerate}
        \item $G$ is a Borel set;
        \item there exists a Borel measure $\mu$ on $Y$ and a compact subset $E$ of $Y$ with $0<\mu(E)<\infty$ and 
        $$\mu\big(Y\backslash(G+y)\big)=0\quad \text{for all } y\in Y.$$
    \end{enumerate}
\end{definition}
\begin{definition}
    For a Borel set $G\subset Y$, a $m$-dimensional subspace $\Lambda$ of $Y$ is said to be a probe if 
    $$\mathscr{L}_\Lambda\big(Y\backslash(G+y)\big)=0\quad\text{for all }y\in Y,$$
    where $\mathscr{L}_\Lambda$ represents $m$-dimensional Lebesgue measure on $\Lambda$ defined in the fundamental way.
\end{definition}
\begin{note}
Here, we highlight some basic terminologies related to the above two definitions:
\begin{enumerate}
    \item A shy set is a complement of a prevalent set.
    \item A set $G$ is referred to as a $m$-prevalent set if it has a $m$-dimensional probe.
    \item The sufficient condition for a set $G$ to be prevalent is to show the existence of a probe.
\end{enumerate}    
\end{note}

Now, we move to the construction of the self-similar fractal. For this, we start with $N(\ge 2)\in \mathbb{N}$ self-similar maps $L_i:\mathbb{R}^n\to \mathbb{R}^n$, $i\in\Sigma(=\{1,2,\dots,N\})$, that are defined as 
$$L_i(\mathbf{x})=a^{(i)}\mathbf{x}+\mathbf{b}^{(i)},$$
where $a^{(i)}\in (-1,1)$ and $\mathbf{b}^{(i)}=(b_1^{(i)},b_2^{(i)},\dots,b_n^{(i)})\in \mathbb{R}^n$. Subsequently, for each $i\in \Sigma$, $L_i$ is the contraction map on $\mathbb{R}^n$ and possesses a unique fixed point, symbolized by $\mathbf{q}^{(i)}$. By the Banach contraction principle, the IFS $\{\mathbb{R}^n;~L_i,i=1,2,\dots,N\}$ yields the unique self-similar fractal $\mathfrak{F}$ that satisfies
$$\mathfrak{F}=\bigcup_{i=1}^N L_i(\mathfrak{F}).$$

\begin{definition}
    Let $N(\ge 2)\in \mathbb{N}$ be fixed. An IFS $\{\mathbb{R}^n;~g_1,g_2,\dots,g_N\}$ is referred to as satisfying the open set condition (OSC) if for some open set $\mathfrak{O}\subseteq \mathbb{R}^n$, we have
    $$\bigcup_{i=1}^N g_i(\mathfrak{O})\subseteq \mathfrak{O}\quad\text{and} \quad g_i(\mathfrak{O})\cap g_j(\mathfrak{O})=\emptyset, \text{ for } i\ne j.$$
\end{definition}
From now onwards, we work under the following assumption:
\begin{assum}\label{assum2.1}
    The self-similar IFS of the above-constructed fractal $\mathfrak{F}$ satisfies the OSC.
\end{assum}

\begin{note}\label{note2.2}
    Under Assumption \ref{assum2.1} and by Hutchinson \cite{JH}, the above-constructed fractal $\mathfrak{F}$ has its box dimension ($\eta$) equal to the similarity dimension ($\mathfrak{s}$), where $\mathfrak{s}$ is the solution of the equation
    $$\sum_{i=1}^N \big(a^{(i)}\big)^\mathfrak{s}=1.$$
\end{note}

\section{Main Results}\label{sec:3}
In this section, we first define the \textit{horizon property} (associated with the box dimension) on $\mathfrak{F} \times \mathfrak{F}$. Next, we construct a completely metrizable topological vector space (on $\mathfrak{F} \times \mathfrak{F}$) to establish the existence of a Borel subset in it, followed by the construction of a probe subspace. This leads us to establishing a prevalent surface that satisfies the horizon property on $\mathfrak{F} \times \mathfrak{F}$. 

For a fixed $m\in \mathbb{N}$, consider
$$\mathcal{C}(\mathfrak{F}^m):=\Big\{h : \mathfrak{F}^m \to \mathbb{R}~|~h \text{ is continuous on } \mathfrak{F}^m\Big\}.$$
Motivated by Falconer and Fraser \cite{FF}, we define the following:
\begin{definition}
    Consider $g \in \mathcal{C}(\mathfrak{F} \times \mathfrak{F})$. The horizon function, $H(g) \in \mathcal{C}(\mathfrak{F})$, of $g$ is defined as 
    $$H(g)(\mathbf{x})= \sup_{\mathbf{y} \in \mathfrak{F} } g(\mathbf{x},\mathbf{y}).$$
\end{definition}
$\bullet$ For any function $h$, we symbolize its graph with $\mathfrak{Gr}(h)$.
\begin{definition}[Horizon Property]
    A function $g\in \mathcal{C}(\mathfrak{F} \times \mathfrak{F})$ exhibits the horizon property (in relation to the box dimension) if the box dimensions of $\mathfrak{Gr}(g)$, $\mathfrak{Gr}(H(g))$ exist and
    $$\dim_B\big(\mathfrak{Gr}(H(g))\big)= \dim_B\big(\mathfrak{Gr}(g)\big)- \eta.$$
\end{definition}

Now, for each $\xi \in [2\eta,2\eta+1]$, let us start by defining a  few subsets as 
\begin{eqnarray*}
    \mathcal{D}_{\xi}(\mathfrak{F} \times \mathfrak{F})&=&\big\{g \in \mathcal{C}(\mathfrak{F} \times \mathfrak{F}):~\overline{\dim}_B\big(\mathfrak{Gr}(g)\big) \le \xi \big\},\\
    \mathcal{B}_{\xi}(\mathfrak{F} \times \mathfrak{F})&=&\big\{g \in \mathcal{C}(\mathfrak{F} \times \mathfrak{F}):~\dim_B\big(\mathfrak{Gr}(g)\big) =  \xi\big\},
\end{eqnarray*}
and 
$$\mathfrak{A}_\xi(\mathfrak{F}\times \mathfrak{F})=\bigg\{g\in \mathcal{B}_\xi(\mathfrak{F}\times \mathfrak{F}):~\xi-\eta\le\underline{\dim}_B\big(\mathfrak{Gr}(H(g))\big)\le \overline{\dim}_B\big(\mathfrak{Gr}(H(g))\big)\le 1+\eta\bigg\}.$$

Next, for a fixed $m\in\mathbb{N}$, we symbolize the collection of all words of the length $m$ by $\Sigma^m$, i.e., if $\omega\in \Sigma^m$ then $\omega=(\omega_1,\omega_2,\dots,\omega_m);~\omega_i\in\Sigma$. For $\omega\in\Sigma^m$, we define
$$L_\omega=L_{\omega_1}\circ L_{\omega_2}\circ \dots \circ L_{\omega_m},$$
and
$$V_m=\{L_\omega(\mathbf{q}^{(i)}):~i\in\Sigma,~\omega\in\Sigma^m\}.$$
Now, for $\omega,\tau\in\Sigma^m$, let us define oscillation of a function $g\in \mathcal{C}(\mathfrak{F} \times \mathfrak{F})$ over $L_\omega(\mathfrak{F})\times L_\tau (\mathfrak{F})$ by
$$\mathcal{R}_g[L_\omega(\mathfrak{F})\times L_\tau (\mathfrak{F})]= \sup\big\{|g(\mathbf{x},\mathbf{y}) -g(\hat{\mathbf{x}},\hat{\mathbf{y}})|:~(\mathbf{x},\mathbf{y}),(\hat{\mathbf{x}},\hat{\mathbf{y}})\in L_\omega(\mathfrak{F})\times L_\tau (\mathfrak{F})\big\},$$
and latter is the total oscillation having  order $m$, 
$$ Osc(m,g)= \sum_{\omega,\tau\in\Sigma^m} \mathcal{R}_g[L_\omega(\mathfrak{F})\times L_\tau (\mathfrak{F})].$$ 
\begin{lemma}\label{lem3.1}
    Let $|a|=\max_{1\leq i \le N}{a^{(i)}}$. Then, for any function $g\in \mathcal{C}(\mathfrak{F} \times \mathfrak{F})$, we have
    \begin{equation}
        \frac{1}{|a|^m}\sum_{\omega,\tau\in \Sigma^m}\mathcal{R}_g[L_\omega(\mathfrak{F})\times L_\tau (\mathfrak{F})]\le N_\delta\big(\mathfrak{Gr}(g)\big)\le 2.N^m.N^m+\frac{1}{|a|^m}\sum_{\omega,\tau\in \Sigma^m}\mathcal{R}_g[L_\omega(\mathfrak{F})\times L_\tau (\mathfrak{F})].\label{osc}
    \end{equation}
\end{lemma}
\begin{proof}
     Let $\mathbf{x},\mathbf{y}\in \mathbb{R}^n$. Consider,
    \begin{eqnarray*}
        \|L_i(\mathbf{x})-L_i(\mathbf{y)}\|_2 &=& |a^{(i)}|~\|\mathbf{x}-\mathbf{y}\|_2\\
        &\le& |a|~\|\mathbf{x}-\mathbf{y}\|_2,
    \end{eqnarray*}
    where $|a|=\max_{1\leq i \le N}{a^{(i)}}$. In general, $$\|L_\omega(\mathbf{x})-L_\omega(\mathbf{y)}\|_2 \le |a|^m~\|\mathbf{x}-\mathbf{y}\|_2.$$
    Now, choose $\delta=|a|^m$ and let $N_\delta\big(\mathfrak{Gr}(g)\big)$ denote the number of convex hulls of side length $\delta$ that intersect with the graph of $g$. Next, on combining this with the continuity of $g$, the required number of convex hulls to cover the part above $L_\omega(\mathfrak{F})\times L_\tau (\mathfrak{F})$ that intersects with $\mathfrak{Gr}(g)$ are at least $\frac{\mathcal{R}_g[L_\omega(\mathfrak{F})\times L_\tau (\mathfrak{F})]}{|a|^m}$ and at most $2+\frac{\mathcal{R}_g[L_\omega(\mathfrak{F})\times L_\tau (\mathfrak{F})]}{|a|^m}$. Adding these overall such parts, we get
    \begin{eqnarray*}
        \sum_{\omega,\tau\in \Sigma^m}\frac{\mathcal{R}_g[L_\omega(\mathfrak{F})\times L_\tau (\mathfrak{F})]}{|a|^m}\le N_\delta\big(\mathfrak{Gr}(g)\big)&\le& 2\sum_{\omega,\tau\in \Sigma^m}1+\sum_{\omega,\tau\in \Sigma^m}\frac{\mathcal{R}_g[L_\omega(\mathfrak{F})\times L_\tau (\mathfrak{F})]}{|a|^m}\\
        \frac{1}{|a|^m}\sum_{\omega,\tau\in \Sigma^m}\mathcal{R}_g[L_\omega(\mathfrak{F})\times L_\tau (\mathfrak{F})]\le N_\delta\big(\mathfrak{Gr}(g)\big)&\le& 2.N^m.N^m+\frac{1}{|a|^m}\sum_{\omega,\tau\in \Sigma^m}\mathcal{R}_g[L_\omega(\mathfrak{F})\times L_\tau (\mathfrak{F})],
    \end{eqnarray*}
    which completes the proof.
\end{proof}

\begin{note}\label{note3.1}
    \begin{enumerate}
        \item For any non-constant function $g$, \eqref{osc} boils down to
        \begin{eqnarray}\label{eqn21}
            N_\delta\big(\mathfrak{Gr}(g)\big)\asymp \frac{1}{|a|^m}\sum_{\omega,\tau\in \Sigma^m}\mathcal{R}_g[L_\omega(\mathfrak{F})\times L_\tau (\mathfrak{F})],
        \end{eqnarray}
        \item For a constant function $g$, we know that $\mathcal{R}_g[L_\omega(\mathfrak{F})\times L_\tau (\mathfrak{F})]=0$ and
        $$\dim_B\big(\mathfrak{Gr}(g)\big)=\dim_B(\mathfrak{F\times F})=2\eta.$$
    \end{enumerate}
\end{note}
In order to avoid the triviality in the results from Proposition \ref{prop23} to Proposition \ref{prop24} and their proofs, we assume that $g+h,g$ and $h$ all are non-constant functions.

\begin{proposition}\label{prop23}
    Consider $g,h \in \mathcal{C}(\mathfrak{F} \times \mathfrak{F})$. 
    \begin{enumerate}
        \item \label{prop23_it1}Then, 
        $$\overline{\dim}_B\big(\mathfrak{Gr}(g+h)\big)\le \max\{\overline{\dim}_B\big(\mathfrak{Gr}(g)\big),\overline{\dim}_B\big(\mathfrak{Gr}(h)\big)\}.$$
        In particular, $\mathcal{D}_\xi(\mathfrak{F} \times \mathfrak{F})$ is a vector space.
        \item \label{prop23_it2}Assume that $\overline{\dim}_B\big(\mathfrak{Gr}(g)\big)\ne \overline{\dim}_B\big(\mathfrak{Gr}(h)\big)$, then 
        $$\overline{\dim}_B\big(\mathfrak{Gr}(g+h)\big)=\max\{\overline{\dim}_B\big(\mathfrak{Gr}(g)\big),\overline{\dim}_B\big(\mathfrak{Gr}(h)\big)\}$$
    \end{enumerate}
\end{proposition}
\begin{proof}
\begin{enumerate}
    \item  Choose $\delta=|a|^m$ and we know that $g$ is a non-constant function. Then, by \eqref{eqn21} of Note \ref{note3.1}, we have 
        \begin{eqnarray*}
            N_\delta\big(\mathfrak{Gr}(g)\big)\asymp \frac{1}{|a|^m}\sum_{\omega,\tau\in \Sigma^m}\mathcal{R}_g[L_\omega(\mathfrak{F})\times L_\tau (\mathfrak{F})],
        \end{eqnarray*}
        that is, there exist constants $\delta_g,K_g>0$ such that for every (sufficiently large $m$) $|a|^m<\delta_g$, we have
        $$\frac{1}{K_g}\le \frac{N_\delta\big(\mathfrak{Gr}(g)\big)}{\frac{1}{|a|^m}\sum\limits_{\omega,\tau\in \Sigma^m}\mathcal{R}_g[L_\omega(\mathfrak{F})\times L_\tau (\mathfrak{F})]} \le K_g.$$
        For $\mathfrak{t}=\max\{\overline{\dim}_B\big(\mathfrak{Gr}(g)\big),\overline{\dim}_B\big(\mathfrak{Gr}(h)\big)\}$, $\epsilon>0$ and \eqref{eqn21},  there exists $\delta_0>0$ such that for every (sufficiently large $m$) $|a|^m<\delta_0$, we get 
        \begin{eqnarray*}
            \sum_{\omega,\tau\in \Sigma^m}\mathcal{R}_g[L_\omega(\mathfrak{F})\times L_\tau (\mathfrak{F})]&\le& \bigg(\frac{1}{|a|^m}\bigg)^{\overline{\dim}_B\big(\mathfrak{Gr}(g)\big)+\epsilon-1}\le \bigg(\frac{1}{|a|^m}\bigg)^{\mathfrak{t}+\epsilon-1},\\
            \sum_{\omega,\tau\in \Sigma^m}\mathcal{R}_h[L_\omega(\mathfrak{F})\times L_\tau (\mathfrak{F})]&\le& \bigg(\frac{1}{|a|^m}\bigg)^{\overline{\dim}_B\big(\mathfrak{Gr}(h)\big)+\epsilon-1}\le \bigg(\frac{1}{|a|^m}\bigg)^{\mathfrak{t}+\epsilon-1}.
        \end{eqnarray*}
        Thus,
        \begin{eqnarray*}
            \frac{1}{|a|^m}\sum_{\omega,\tau\in \Sigma^m}\mathcal{R}_{g+h}[L_\omega(\mathfrak{F})\times L_\tau (\mathfrak{F})]&\le& \frac{1}{|a|^m}\sum_{\omega,\tau\in \Sigma^m}\mathcal{R}_g[L_\omega(\mathfrak{F})\times L_\tau (\mathfrak{F})]+\frac{1}{|a|^m}\sum_{\omega,\tau\in \Sigma^m}\mathcal{R}_h[L_\omega(\mathfrak{F})\times L_\tau (\mathfrak{F})]\\
            &\le& 2~\bigg(\frac{1}{|a|^m}\bigg)^{\mathfrak{t}+\epsilon},
        \end{eqnarray*}
        implies that $\overline{\dim}_B\big(\mathfrak{Gr}(g+h)\big)\le \mathfrak{t}+\epsilon$. As it is true for all $\epsilon>0$, $\overline{\dim}_B\big(\mathfrak{Gr}(g+h)\big)\le \mathfrak{t}$, completes the assertion.
        \item We prove this part by contradiction. Thus, for any functions $g,h \in \mathcal{C}(\mathfrak{F} \times \mathfrak{F})$, without loss of generality, assume that $\overline{\dim}_B\big(\mathfrak{Gr}(g)\big)< \overline{\dim}_B\big(\mathfrak{Gr}(h)\big)$. Let $f\in \mathcal{C}(\mathfrak{F} \times \mathfrak{F})$ with $f=g+h$, where $\overline{\dim}_B\big(\mathfrak{Gr}(f)\big)\ne \overline{\dim}_B\big(\mathfrak{Gr}(h)\big)$. Then, by Item \ref{prop23_it1}, $\overline{\dim}_B\big(\mathfrak{Gr}(f)\big)< \overline{\dim}_B\big(\mathfrak{Gr}(h)\big)$, which is a contradiction as
        \begin{eqnarray*}
            \overline{\dim}_B\big(\mathfrak{Gr}(f-g)\big)&=&\overline{\dim}_B\big(\mathfrak{Gr}(h)\big)\\
            &>& \max\{\overline{\dim}_B\big(\mathfrak{Gr}(f)\big),\overline{\dim}_B\big(\mathfrak{Gr}(-g)\big)\}
        \end{eqnarray*}
\end{enumerate}
\end{proof}
\begin{remark}\label{rmk3.1}
    Let $k\in \mathbb{R}\backslash \{0\}$. If $\overline{\dim}_B\big(\mathfrak{Gr}(g)\big)=\xi$, then $$\overline{\dim}_B\big(\mathfrak{Gr}(kg)\big)=\xi.$$
    Moreover, for $k=0$, then $\overline{\dim}_B\big(\mathfrak{Gr}(\mathbf{0})\big)=2\eta.$
\end{remark}
\begin{proof}
    For $k\ne 0$, let us define a map $\mathcal{T}: \mathfrak{Gr}(g)\to \mathfrak{Gr}(kg)$ as:
    $$\mathcal{T}(\mathbf{x},\mathbf{y},g(\mathbf{x},\mathbf{y}))=(\mathbf{x},\mathbf{y},kg(\mathbf{x},\mathbf{y})).$$
    It is easy to see that $\mathcal{T}$ is a bijection and a bi-Lipschitz map, which proves the assertion.
\end{proof}
\begin{proposition}\label{prop24}
    Let $g,h \in \mathcal{C}(\mathfrak{F}\times \mathfrak{F})$.
    \begin{enumerate}
        \item For each $k\in \mathbb{R}$ with two exceptions of $k=0$ and another value of $k$, we get
        $$\overline{\dim}_B\big(\mathfrak{Gr}(g+kh)\big)=\max\{\overline{\dim}_B\big(\mathfrak{Gr}(g)\big),\overline{\dim}_B\big(\mathfrak{Gr}(h)\big)\}.$$
        \item For $\mathscr{L}^1$-almost all $k\in \mathbb{R}$, we have
        $$\underline{\dim}_B\big(\mathfrak{Gr}(g+h)\big)\ge \max\{\underline{\dim}_B\big(\mathfrak{Gr}(g)\big),\underline{\dim}_B\big(\mathfrak{Gr}(h)\big)\},$$
        here $\mathscr{L}^1$ denotes the Lebesgue measure on $\mathbb{R}$.
    \end{enumerate}
    In particular, for $\mathscr{L}^1$-almost all $k\in \mathbb{R}$, one can have
    $$\max\{\underline{\dim}_B\big(\mathfrak{Gr}(g)\big),\underline{\dim}_B\big(\mathfrak{Gr}(h)\big)\}\le \underline{\dim}_B\big(\mathfrak{Gr}(g+kh)\big)\le \overline{\dim}_B\big(\mathfrak{Gr}(g+kh)\big)=\max\{\overline{\dim}_B\big(\mathfrak{Gr}(g)\big),\overline{\dim}_B\big(\mathfrak{Gr}(h)\big)\}.$$
\end{proposition}
\begin{proof}
    \begin{enumerate}
        \item This part can be proven from Proposition \ref{prop23}.
        \item Consider $g,h\in \mathcal{C}(\mathfrak{F}\times \mathfrak{F})$ and $\epsilon>0$. Without loss of generality, assume that $\max\{\underline{\dim}_B\big(\mathfrak{Gr}(g)\big),\underline{\dim}_B\big(\mathfrak{Gr}(h)\big)\}=\underline{\dim}_B\big(\mathfrak{Gr}(h)\big)=\mathfrak{t}$. By \eqref{eqn21}, there exists $\delta_0>0$ such that for every (sufficiently large $m$) $|a|^m<\delta_0$, we get 
        \begin{eqnarray}
            \sum_{\omega,\tau\in \Sigma^m}\mathcal{R}_h[L_\omega(\mathfrak{F})\times L_\tau (\mathfrak{F})]&\ge& \bigg(\frac{1}{|a|^m}\bigg)^{\overline{\dim}_B\big(\mathfrak{Gr}(h)\big)-\epsilon-1}= \bigg(\frac{1}{|a|^m}\bigg)^{\mathfrak{t}-\epsilon-1}.\label{eqnn3.3}
        \end{eqnarray}
        Suppose that $Q$ is any bounded Lebesgue measurable subset of $\mathbb{R}$. Then, for every $a,b\in \mathbb{R}$, we know that 
        \begin{eqnarray}
            \int_Q |a-kb|~dk\ge \frac{1}{4}\big(\mathscr{L}^1(Q)\big)^2~|b|.\label{eqnn3.4}
        \end{eqnarray}
        Fix $|a|^m<\delta_0$. Then, using \eqref{eqnn3.3} and \eqref{eqnn3.4},
        \begin{eqnarray}
            \int_Q~\sum_{\omega,\tau\in \Sigma^m}\mathcal{R}_{g+kh}[L_\omega(\mathfrak{F})\times L_\tau (\mathfrak{F})]~dk &\ge& \sum_{\omega,\tau\in \Sigma^m} \int_Q~\Big{|}\mathcal{R}_{g}[L_\omega(\mathfrak{F})\times L_\tau (\mathfrak{F})]-k\mathcal{R}_{h}[L_\omega(\mathfrak{F})\times L_\tau (\mathfrak{F})]\Big{|}~dk\nonumber\\
            &\ge& \sum_{\omega,\tau\in \Sigma^m}\frac{1}{4}\big(\mathscr{L}^1(Q)\big)^2~\Big{|}\mathcal{R}_{h}[L_\omega(\mathfrak{F})\times L_\tau (\mathfrak{F})]\Big{|}\nonumber\\
            &\ge& \frac{1}{4}\big(\mathscr{L}^1(Q)\big)^2~\bigg(\frac{1}{|a|^m}\bigg)^{\mathfrak{t}-\epsilon-1}\label{eqnn3.5}
        \end{eqnarray}
        Consider $K\in \mathbb{N}$ and 
        $$Q_m^K=\Bigg\{k\in[-K,K]:~\sum_{\omega,\tau\in \Sigma^m}\mathcal{R}_{g+kh}[L_\omega(\mathfrak{F})\times L_\tau (\mathfrak{F})]\le \bigg(\frac{1}{|a|^m}\bigg)^{\mathfrak{t}-2\epsilon-1}\Bigg\}.$$
        From \eqref{eqnn3.5}, it is easy to see that
        $$\mathscr{L}^1\big(Q^K_m\big)~\bigg(\frac{1}{|a|^m}\bigg)^{\mathfrak{t}-2\epsilon-1}\ge \int_{Q_m^K}~\sum_{\omega,\tau\in \Sigma^m}\mathcal{R}_{g+kh}[L_\omega(\mathfrak{F})\times L_\tau (\mathfrak{F})]~dk\ge \frac{1}{4}\Big(\mathscr{L}^1\big(Q_m^K\big)\Big)^2~\bigg(\frac{1}{|a|^m}\bigg)^{\mathfrak{t}-\epsilon-1}.$$
        Thus, $\mathscr{L}^1\big(Q_m^K\big)\le 4(|a|^m)^\epsilon$. Further, 
        $$\sum_{m\ge M}\mathscr{L}^1\big(Q_m^K\big)\le 4\sum_{m\ge M}(|a|^m)^\epsilon<\infty,$$
        and Borel-Cantelli Lemma yields,
        $$\mathscr{L}^1\Bigg(\bigcap_{K\in \mathbb{N}}\bigcup_{m\ge K} Q^K_m\Bigg)=0,$$
        i.e., for $\mathscr{L}^1$-almost all $k\in[-K,K]$, $k\notin Q_{m}$ for sufficiently large $m$. Therefore, for $\mathscr{L}^1$-almost all $k\in[-K,K]$, there exists $\delta_k>0$ such that $\delta<\delta_k$
        $$\sum_{\omega,\tau\in \Sigma^m}\mathcal{R}_{g+kh}[L_\omega(\mathfrak{F})\times L_\tau (\mathfrak{F})]>\bigg(\frac{1}{|a|^m}\bigg)^{\mathfrak{t}-2\epsilon-1}.$$
        Hence, $$\underline{\dim}_B\big(\mathfrak{Gr}(g+kh)\big)\ge \mathfrak{t}-2\epsilon,$$
        since $\mathbb{R}=\bigcup\limits_{K\in\mathbb{N}}[-K,K]$ and the above inequality is true for small $\epsilon$, the required assertion follows.
    \end{enumerate}
\end{proof}
\begin{note} \label{note25}
    In a similar way, one may also prove Proposition \ref{prop23} and Proposition \ref{prop24} for $\mathfrak{F}^m~(1\le m\le N)$.
\end{note}

To scrutinize the prevalent subsets of $\mathcal{D}_{\xi}(\mathfrak{F} \times \mathfrak{F})$, we must demonstrate that $\mathcal{D}_{\xi}(\mathfrak{F} \times \mathfrak{F})$ is a completely metrizable topological vector space. The part of showing $\mathcal{D}_{\xi}(\mathfrak{F} \times \mathfrak{F})$ a vector space has already been covered in Proposition \ref{prop23}, and now we shall build an appropriate metric. Thus, let us define the oscillation space $\mathbb{O}^{\xi}(\mathfrak{F} \times \mathfrak{F}),\text{ for}~\xi \in [2\eta,2\eta+1]$, by 

$$ \mathbb{O}^{\xi}(\mathfrak{F} \times \mathfrak{F})= \bigg\{g \in \mathcal{C}(\mathfrak{F} \times \mathfrak{F}): \sup_{m \in \mathbb{N} }\dfrac{ Osc(m,g)}{|a|^{-m\left(\xi-1\right)}}< \infty \bigg\}.$$

\begin{remark}\label{rmk3.2}
    In \cite{FF}, the oscillation space on $[0,1]^d$ is defined by imposing $ \|g\|_{\mathbb{O}^{\xi}}< \infty,$
        where $\|g\|_{\mathbb{O}^{\xi}}:= \|g\|_{\infty} + \sup_{m \in \mathbb{N}}\dfrac{ Osc(m,g)}{|a|^{-m\left(\xi-1\right)}}$. While we define oscillation space by only imposing $\sup_{m \in \mathbb{N}}\dfrac{ Osc(m,g)}{|a|^{-m\left(\xi-1\right)}}<\infty$.  This is because $g \in \mathcal{C}(\mathfrak{F} \times \mathfrak{F})$, which implies that $\|g\|_\infty<\infty$. Thus, $ \|g\|_{\mathbb{O}^{\xi}}< \infty$ boils down to $\sup_{m \in \mathbb{N}}\dfrac{ Osc(m,g)}{|a|^{-m\left(\xi-1\right)}}<\infty$.
\end{remark}
We are recalling one lemma that can be proven via applying the simply definition of $Osc(m,f)$.
\begin{lemma}\label{lem3.2}
    Let $c\in\mathbb{R}$ and let $f_1,f_2\in\mathcal{C}(\mathfrak{F} \times \mathfrak{F})$. Then, for a fixed $m\in\mathbb{N}$, we get:
    \begin{enumerate}
        \item $Osc(m,cf_1)=|c|~Osc(m,f_1).$
        \item  $Osc(m,f_1+f_2)\le Osc(m,f_1)+Osc(m,f_2).$
        \item $Osc(m,f_1f_2)\leq\|f_1\|_{\infty}~Osc(m,f_2)+\|f_2\|_{\infty}~Osc(m,f_1).$
    \end{enumerate}
\end{lemma}
\begin{lemma}\label{lem3.3}
    Consider a sequence $\{g_n\}_{n=1}^\infty$ from $\mathcal{C}(\mathfrak{F}\times\mathfrak{F})$, which is uniformly convergent to a function $g:\mathfrak{F\times F}\to \mathbb{R}$. Then,
    $$Osc(m,g_n)\to Osc(m,g).$$
    Furthermore, if this uniform convergence is in $\mathbb{O}^{\xi}(\mathfrak{F} \times \mathfrak{F})$, then
    $$\sup_{m\in\mathbb{N}}\dfrac{ Osc(m,g)}{|a|^{-m\left(\xi-1\right)}}\le  \liminf_{n\to \infty}\sup_{m\in\mathbb{N}}\dfrac{ Osc(m,g_n)}{|a|^{-m\left(\xi-1\right)}}.$$
\end{lemma}
\begin{proof}
    Let $\Upsilon_n=\sup_{(\mathbf{x},\mathbf{y})\in \mathfrak{F} \times \mathfrak{F}}|g_n(\mathbf{x},\mathbf{y}) -g(\mathbf{x},\mathbf{y})|$. By the uniform convergence of $\{g_n\}$, it is evident that $\Upsilon_n\to 0$ as $n\to\infty$. Let $(\mathbf{x},\mathbf{y}),(\widehat{\mathbf{x}},\widehat{\mathbf{y}})\in L_\omega(\mathfrak{F})\times L_\tau(\mathfrak{F})$, then
    \begin{eqnarray*}
        |g_n(\mathbf{x},\mathbf{y})-g_n(\widehat{\mathbf{x}},\widehat{\mathbf{y}})|&\le&|g(\mathbf{x},\mathbf{y})-g(\widehat{\mathbf{x}},\widehat{\mathbf{y}})|+|g_n(\mathbf{x},\mathbf{y})-g(\mathbf{x},\mathbf{y})|+|g_n(\widehat{\mathbf{x}},\widehat{\mathbf{y}})-g(\widehat{\mathbf{x}},\widehat{\mathbf{y}})|\\
        &\le& |g(\mathbf{x},\mathbf{y})-g(\widehat{\mathbf{x}},\widehat{\mathbf{y}})|+2\Upsilon_n.
    \end{eqnarray*}
    On combining this with the definition of $\mathcal{R}_g$ over $L_\omega(\mathfrak{F})\times L_\tau(\mathfrak{F})$, we have
    \begin{eqnarray*}
        \mathcal{R}_{g_n}[L_\omega(\mathfrak{F})\times L_\tau(\mathfrak{F})]\le \mathcal{R}_{g}[L_\omega(\mathfrak{F})\times L_\tau(\mathfrak{F})]+2\Upsilon_n.
    \end{eqnarray*}
    By reversing the role of $g$ and $g_n$, we get
    \begin{eqnarray*}
        \mathcal{R}_{g}[L_\omega(\mathfrak{F})\times L_\tau(\mathfrak{F})]\le \mathcal{R}_{g_n}[L_\omega(\mathfrak{F})\times L_\tau(\mathfrak{F})]+2\Upsilon_n.
    \end{eqnarray*}
    Thus,
    \begin{eqnarray*}
        |\mathcal{R}_{g_n}[L_\omega(\mathfrak{F})\times L_\tau(\mathfrak{F})]- \mathcal{R}_{g}[L_\omega(\mathfrak{F})\times L_\tau(\mathfrak{F})]|\le 2\Upsilon_n.
    \end{eqnarray*}
    Now, consider
    \begin{eqnarray*}
        |Osc (m,g_n)- Osc (m,g)| &=& \Bigg{|}\sum_{\omega,\tau\in\Sigma^m}\mathcal{R}_{g_n}[L_\omega(\mathfrak{F})\times L_\tau(\mathfrak{F})]- \sum_{\omega,\tau\in\Sigma^m}\mathcal{R}_{g}[L_\omega(\mathfrak{F})\times L_\tau(\mathfrak{F})]\Bigg{|} \\ 
                     &\le& 2\Upsilon_n\sum_{\omega,\tau\in\Sigma^m} 1\\
                     &=& 2.N^m. N^m\Upsilon_n\\
                     &\to& 0,
    \end{eqnarray*}
    as $n \to \infty$.
    For the next part, assume that $\{g_n\}_{n=1}^\infty$ converges uniformly to $g$ in $\mathbb{O}^\xi(\mathfrak{F} \times \mathfrak{F})$ and on combining this with the above result, we get
    \begin{eqnarray*}
        \frac{Osc(m,g)}{|a|^{-m\left(\xi-1\right)}}&=&\lim_{n\to \infty} \frac{Osc(m,g_n)}{|a|^{-m\left(\xi-1\right)}}\\
        &\le& \sup_{m\in\mathbb{N}}\lim_{n\to \infty} \frac{Osc(m,g_n)}{|a|^{-m\left(\xi-1\right)}}\\
        &\le& \liminf_{n\to \infty}\sup_{m\in\mathbb{N}} \frac{Osc(m,g_n)}{|a|^{-m\left(\xi-1\right)}},
    \end{eqnarray*}
    for every $m\in \mathbb{N}$. Therefore, on taking the supremum over $m$ on the left-hand side, we get the required inequality.
\end{proof}
\begin{proposition}
    For each $\xi \in [2\eta,2\eta+1]$, the space $\mathbb{O}^{\xi}(\mathfrak{F}\times \mathfrak{F})$ equipped with the norm 
    $$\|g\|_{\mathbb{O}^{\xi}}:= \|g\|_{\infty} + \sup_{m \in \mathbb{N}}\dfrac{ Osc(m,g)}{|a|^{-m\left(\xi-1\right)}},$$
    is a Banach space. 
\end{proposition}
\begin{proof}
    Using Lemma \ref{lem3.2}, it is easy to prove that $\big(\mathbb{O}^{\xi}(\mathfrak{F}\times \mathfrak{F}),\|.\|_{\mathbb{O}^\xi}\big)$ is a vector space. Now, consider $\{g_n\}_{n=1}^{\infty}$ is a Cauchy sequence in $\mathbb{O}^{\xi}(\mathfrak{F}\times \mathfrak{F})$. Then, for each $\epsilon>0$, there exists a natural no. $K$ with
    $$\|g_n-g_k\|_{\mathbb{O}^{\xi}}<\epsilon\quad\forall~n,k\ge K.$$
    It is also evident that $\|g_n-g_k\|_{\infty}<\epsilon\quad\forall~n,k\ge K$. Now, the completeness of $\mathcal{C}(\mathfrak{F\times F})$ w.r.t. $\|.\|_\infty$ guarantees the existence of a function $g\in \mathcal{C}(\mathfrak{F\times F})$ such that $\{g_n\}_{n=1}^{\infty}$ uniformly converges to $g$. Now, let $n\ge K$. By Lemma \ref{lem3.3}, we get
    \begin{eqnarray*}
        \|g_n-g\|_{\mathbb{O}^{\xi}}&=& \|g_n-g\|_{\infty} + \sup_{m \in \mathbb{N}}\dfrac{ Osc(m,g_n-g)}{|a|^{-m\left(\xi-1\right)}}\\
        &=& \lim_{k\to \infty} \Bigg(\|g_n-g_k\|_{\infty} + \sup_{m \in \mathbb{N}}\dfrac{ Osc(m,g_n-g_k)}{|a|^{-m\left(\xi-1\right)}}\Bigg)\\
        &=& \sup_{k\ge K} \Bigg(\|g_n-g_k\|_{\infty} + \sup_{\hat{m} \in \mathbb{N}}\dfrac{ Osc(\hat{m},g_n-g_k)}{|a|^{-\hat{m}\left(\xi-1\right)}}\Bigg)\\
        &=&  \sup_{k\ge K} \|g_n-g_k\|_{\mathbb{O}^{\xi}}\\
        &\le& \epsilon.
    \end{eqnarray*}
    Thus, $g_n-g\in \mathbb{O}^{\xi}(\mathfrak{F}\times \mathfrak{F})$ for $n\ge K$. Also, $g_K-g\in \mathbb{O}^{\xi}(\mathfrak{F}\times \mathfrak{F})$. By Lemma \ref{lem3.2}, $g=g_K-g-g_K\in \mathbb{O}^{\xi}(\mathfrak{F}\times \mathfrak{F})$, which completes the proof. 
\end{proof}
Before establishing $\mathcal{D}_{\xi}(\mathfrak{F} \times \mathfrak{F})$ to be a completely metrizable topological vector space, let us recall the following result:
\begin{lemma}\cite{FF}\label{lem3.4}
    Let $(Y_n,\|.\|_n)$ be a decreasing sequence of Banach spaces, i.e., for each $n\in \mathbb{N}$, one can have $Y_n\ge Y_{n+1}$ and for $y\in Y_{n+1}$ one can have $\|y\|_{n+1}\ge\|y\|_n$. Then, $\bigcap_{n\in\mathbb{N}}Y_n$ is a complete metric space w.r.t. the metric $$\Theta(y_1,y_2)=\sum_{n=1}^\infty\min\big\{2^{-n},\|y_1-y_2\|_{n}\big\}.$$
\end{lemma}
\begin{proposition}\label{prop3.4}
    For each $\xi \in [2\eta,2\eta+1]$, the space $\mathcal{D}_\xi(\mathfrak{F}\times \mathfrak{F})$ is a Banach space corresponding to the metric
    $$\Theta_\xi(g,h)=\sum_{n=1}^\infty\min\big\{2^{-n},\|g-h\|_{\mathbb{O}^{\xi+\frac{1}{n}}}\big\}.$$
\end{proposition}
\begin{proof}
    It is clear from the definition of $\|.\|_{{\mathbb{O}^{\xi}}}$ that $\left(\mathbb{O}_{\xi+\frac{1}{n}}(\mathfrak{F}\times \mathfrak{F}),\|.\|_{\mathbb{O}^{\xi+\frac{1}{n}}}\right)$ is a decreasing sequence of Banach spaces. Then, by Lemma \ref{lem3.4}, $\left(\bigcap_{n\in \mathbb{N}}\mathbb{O}^{\xi+\frac{1}{n}}(\mathfrak{F}\times \mathfrak{F}),\Theta_\xi\right)$ is a Banach space. Now, we just need to show that $\mathcal{D}_\xi(\mathfrak{F}\times \mathfrak{F})=\bigcap_{n\in \mathbb{N}}\mathbb{O}^{\xi+\frac{1}{n}}(\mathfrak{F}\times \mathfrak{F})$. For this, let $g\in \mathcal{D}_\xi(\mathfrak{F}\times \mathfrak{F})$, then 
    $\overline{\dim}_B\big(\mathfrak{Gr}(g)\big) \le \xi$; consequently, by \eqref{eqn21}, for each $n\in \mathbb{N}$, there exists $\delta_g$ such that for every (sufficiently large $m$) $|a|^m<\delta_g$, we get
    $$\frac{1}{|a|^m}\sum_{\omega,\tau\in \Sigma^m}\mathcal{R}_g[L_\omega(\mathfrak{F})\times L_\tau (\mathfrak{F})]\le N_\delta\big(\mathfrak{Gr}(g)\big)\le \left(\frac{1}{|a|^m}\right)^{\xi+\frac{1}{n}}.$$
    Therefore, 
    \begin{eqnarray*}
        \sup_{m\in\mathbb{N}}\frac{Osc(m,g)}{|a|^{-m\left(\xi+\frac{1}{n}-1\right)}}&\le& 1\\
        &<&\infty.
    \end{eqnarray*}
    Hence, $g\in \bigcap_{n\in \mathbb{N}}\mathbb{O}^{\xi+\frac{1}{n}}(\mathfrak{F}\times \mathfrak{F})$. Next, for the other side of inclusion, let $g\in \bigcap_{n\in \mathbb{N}}\mathbb{O}^{\xi+\frac{1}{n}}(\mathfrak{F}\times \mathfrak{F})$. Then, for each $n\in\mathbb{N}$, we have $C_{g,n}=\sup_{m\in\mathbb{N}}\frac{Osc(m,g)}{|a|^{-m\left(\xi+\frac{1}{n}-1\right)}}<\infty,$ where $C_{g,n}$ depends only on $g$ and $n$.
    Thus, 
    \begin{eqnarray*}
        \frac{1}{|a|^{m}}\sum_{\omega,\tau\in \Sigma^m}\mathcal{R}_g[L_\omega(\mathfrak{F})\times L_\tau (\mathfrak{F})]\le N_\delta\big(\mathfrak{Gr}(g)\big)\le C_{g,n} \left(\frac{1}{|a|^{m}}\right)^{\xi+\frac{1}{n}},
    \end{eqnarray*}
    for some constant $C_{g,n}>0$. Now, on combining the above inequality with the definition of the upper box dimension, we get $\overline{\dim}_B\big(\mathfrak{Gr}(g)\big) \le \xi+\frac{1}{n}$ for all $n$. Hence, $\overline{\dim}_B\big(\mathfrak{Gr}(g)\big) \le \xi$, and subsequently $g\in \mathcal{D}_\xi(\mathfrak{F\times F})$.
\end{proof}
The continuity of the vector space operations w.r.t. the topology induced by $\Theta_\xi$ allows to say that $(\mathcal{D}_\xi(\mathfrak{F}\times \mathfrak{F}), \Theta_\xi)$ is a completely metrizable topological vector space. Therefore, now we can scrutinize the prevalent subsets of $\mathcal{D}_\xi(\mathfrak{F}\times \mathfrak{F})$.

\begin{lemma}\label{lem23}
\begin{enumerate}
    \item \label{lem23_it1}If $\mathfrak{O}$ is an open set in $\Theta_\infty$, then $\mathfrak{O}$ is also open in $\Theta_\xi$, where $\Theta_\infty(f,g)=\|f-g\|_\infty$.
    \item \label{lem23_it2}For $h,g\in \mathcal{C}(\mathfrak{F}\times \mathfrak{F})$, $h\in \mathcal{U}_{\Theta_\infty}(g,r)$ implies that $H(h)\in \mathcal{U}_{\Theta_\infty}(H(g),r)$, where $\mathcal{U}_{\Theta_\infty}(\mathbf{a},r)$ denotes the open ball centered at $\mathbf{a}$ with radius $r>0$ in $\big(\mathcal{C}(\mathfrak{F} \times \mathfrak{F}),\Theta_\infty\big).$
\end{enumerate}
\end{lemma}
\begin{proof}
    \begin{enumerate}
        \item It is sufficient to show that 
              $\mathfrak{O}^c$ is closed, i.e., for a sequence $f_n\in \mathfrak{O}^c$ such that $f_n\to f$ w.r.t. $\Theta_\xi$, then $f\in \mathfrak{O}^c$.
              By the definition of $\Theta_\xi$, we have
              \begin{eqnarray*}
                  \Theta_\xi(f_n,f)&\to& 0 \quad \text{ as } n\to \infty\\
                  \Theta_\infty(f_n,f)&\to& 0 \quad \text{ as } n\to \infty.
              \end{eqnarray*}
             Thus, $f\in \mathfrak{O}^c$ in $\Theta_\infty$. Hence, $f\in \mathfrak{O}^c$ in $\Theta_\xi$, completes the assertion.
             \item For a fixed $\mathbf{x}\in \mathfrak{F}$, using the inequality $\sup f -\sup g \le \sup(f-g)$ we get
             \begin{eqnarray*}
                 H(h)(\mathbf{x})-H(g)(\mathbf{x}) &=& \sup_{\mathbf{y}\in \mathfrak{F}}h(\mathbf{x},\mathbf{y})-\sup_{\mathbf{y}\in \mathfrak{F}}g(\mathbf{x},\mathbf{y})\\
                 &\le& \sup_{\mathbf{y}\in \mathfrak{F}}(h(\mathbf{x},\mathbf{y})-g(\mathbf{x},\mathbf{y}))\\&\le& \sup_{\mathbf{x}, \mathbf{y}\in \mathfrak{F}}|h(\mathbf{x},\mathbf{y})-g(\mathbf{x},\mathbf{y})|\\
                 &\le& \|h(\mathbf{x},\mathbf{y})-g(\mathbf{x},\mathbf{y})\|_\infty< r,
             \end{eqnarray*}
             and \begin{eqnarray*}
                 H(g)(\mathbf{x})-H(h)(\mathbf{x}) &=& \sup_{\mathbf{y}\in \mathfrak{F}}g(\mathbf{x},\mathbf{y})-\sup_{\mathbf{y}\in \mathfrak{F}}g(\mathbf{x},\mathbf{y})\\
                 &\le& \sup_{\mathbf{y}\in \mathfrak{F}}(g(\mathbf{x},\mathbf{y})-h(\mathbf{x},\mathbf{y}))\\&\le& \sup_{\mathbf{x}, \mathbf{y}\in \mathfrak{F}}|h(\mathbf{x},\mathbf{y})-g(\mathbf{x},\mathbf{y})|\\
                 &\le& \|h(\mathbf{x},\mathbf{y})-g(\mathbf{x},\mathbf{y})\|_\infty< r,
             \end{eqnarray*}
             Since the last inequality true for $\mathbf{x}\in \mathfrak{F}$, yields the assertion.
     \end{enumerate}
\end{proof}
\begin{proposition}\label{prop26}
     Under the above defined $\mathfrak{A}_\xi(\mathfrak{F}\times \mathfrak{F})$ and $\mathcal{D}_\xi(\mathfrak{F}\times \mathfrak{F})$, we have
    \begin{enumerate}
        \item \label{prop26_it1} for each $\xi\in [2\eta,2\eta+1)$, the set $\mathfrak{A}_\xi(\mathfrak{F}\times \mathfrak{F})$ is a Borel subset of $(\mathcal{D}_\xi(\mathfrak{F}\times \mathfrak{F}), \Theta_\xi)$;
        \item \label{prop26_it2} $\mathfrak{A}_{2\eta+1}(\mathfrak{F}\times \mathfrak{F})$ is a Borel subset of $(\mathcal{C}(\mathfrak{F}\times \mathfrak{F}), \Theta_\infty)$.
    \end{enumerate}
\end{proposition}
\begin{proof}
    Let us first start with the proof of Item (\ref{prop26_it1}).
    \begin{enumerate}
        \item For $\xi\in [2\eta,2\eta+1)$, we get
        \begin{eqnarray*}
            \mathfrak{A}_\xi(\mathfrak{F}\times \mathfrak{F})=\mathcal{B}_\xi(\mathfrak{F}\times \mathfrak{F})\cap \mathfrak{A}_2,
        \end{eqnarray*} 
        where $\mathcal{B}_\xi(\mathfrak{F}\times \mathfrak{F})$ is same as defined above and
        $$\mathfrak{A}_2:=\bigg\{g\in \mathcal{C}(\mathfrak{F}\times \mathfrak{F}):~\xi-\eta\le\underline{\dim}_B\big(\mathfrak{Gr}(H(g))\big)\le \overline{\dim}_B\big(\mathfrak{Gr}(H(g))\big)\le 1+\eta\bigg\}.$$
        
        Using \eqref{eqn21}, we may write $$\mathcal{B}_\xi(\mathfrak{F}\times \mathfrak{F})=\bigcap_{r\in\mathbb{Q}^+}\bigcup_{K\in \mathbb{N}}\bigcap_{n\ge K}\mathfrak{A}_1^n\quad\text{and}\quad \mathfrak{A}_2=\bigcap_{r\in\mathbb{Q}^+}\bigcup_{K\in \mathbb{N}}\bigcap_{n\ge K}\mathfrak{A}_2^{n},$$
        where $\mathbb{Q}^+$ represents the collection of all positive rational numbers and 
        $$\mathfrak{A}_1^n=\bigg\{g\in \mathcal{C}(\mathfrak{F}\times \mathfrak{F}):~2^{(\xi-r-1)n}<\sum_{\omega,\tau\in \Sigma^m}\mathcal{R}_g[L_\omega(\mathfrak{F})\times L_\tau (\mathfrak{F})]<2^{(\xi+r-1)n}\bigg\},$$
        and
        $$\mathfrak{A}_2^n=\bigg\{g\in \mathcal{C}(\mathfrak{F}\times \mathfrak{F}):~2^{\big(\xi-\eta-r-1\big)n}<\sum_{\omega,\tau\in \Sigma^m}\mathcal{R}_{H(g)}[L_\omega(\mathfrak{F})\times L_\tau (\mathfrak{F})]<2^{\big(1+\eta+r-1\big)n}\bigg\},$$
        It is clear that for each $r,K,n$, the set $\mathfrak{A}_1^n$ is open w.r.t. $\Theta_\infty$, and using Item (\ref{lem23_it2}) of Lemma \ref{lem23}, the set $\mathfrak{A}_2^n$ is also open w.r.t. $\Theta_\infty$. From Item (\ref{lem23_it1}) of Lemma \ref{lem23}, $\mathcal{B}_\xi(\mathfrak{F}\times \mathfrak{F}),\mathfrak{A}_1$ are also open sets w.r.t. $\Theta_\xi$. Hence, $\mathfrak{A}_\xi(\mathfrak{F}\times \mathfrak{F})$ is a Borel subset of $\big(\mathcal{D}_\xi(\mathfrak{F}\times \mathfrak{F}),\Theta_\xi\big)$. 
        \item One can prove this part following the similar lines of Item (\ref{prop26_it1}) of this proposition and \cite[Lemma 4.1]{FF}.
    \end{enumerate}
\end{proof}

\begin{note}\label{note3.3}
    Here, we define a map to construct the probe subspace. For $\xi\in [2\eta,2\eta+1]$, choose a non-constant function $\theta_\xi\in \mathcal{B}_{\xi-\eta}(\mathfrak{F})$. Then, the function $\Phi_\xi: \mathfrak{F} \times \mathfrak{F} \to\mathbb{R}$, defined by
    $$\Phi_\xi(\mathbf{x},\mathbf{y})=\theta_\xi(\mathbf{x}),$$
    has box dimension $\xi$, i.e., $\dim_B\big(\mathfrak{Gr}(\Phi_\xi)\big)=\xi.$ It is simple to deduce that $H(\Phi_\xi)=\theta_\xi$ and $\Phi_\xi\in \mathcal{D}_\xi(\mathfrak{F}\times \mathfrak{F})$.
\end{note}

\begin{proposition}\label{prop28}
    Let $\xi\in [2\eta,2\eta+1]$ and let $g\in \mathcal{D}_\xi(\mathfrak{F}\times \mathfrak{F})$. For $\mathscr{L}^1$-almost all $k\in \mathbb{R}$, we get
    \begin{enumerate}
        \item $\dim_B\big(\mathfrak{Gr}(g+k\Phi_\xi)\big)=\dim_B\big(\mathfrak{Gr}(\Phi_\xi)\big)$.
        \item $\xi-\eta\le \underline{\dim}_B\big(\mathfrak{Gr}(H(g+k\Phi_\xi))\big)\le \overline{\dim}_B\big(\mathfrak{Gr}(H(g+k\Phi_\xi))\big)\le 1+\eta$.
    \end{enumerate}
\end{proposition}
\begin{proof}
    \begin{enumerate}
        \item From Proposition \ref{prop24}, one may deduce the assertion of this part easily.
        \item Independence of $\Phi_\xi(\mathbf{x},\mathbf{y})$ w.r.t. $\mathbf{y}$ yields
        \begin{eqnarray*}
            H(g+k\Phi_\xi)(\mathbf{x})&=& \sup_{\mathbf{y}\in \mathfrak{F}} (g+k\Phi)(\mathbf{x},\mathbf{y})\\
            &=& \sup_{\mathbf{y}\in \mathfrak{F}} \big(g(\mathbf{x},\mathbf{y})+k\Phi(\mathbf{x},\mathbf{y})\big)\\
            &=& \sup_{\mathbf{y}\in \mathfrak{F}} \big(g(\mathbf{x},\mathbf{y})\big)+kH(\Phi)(\mathbf{x})\\
            &=& H(g)(\mathbf{x})+kH(\Phi)(\mathbf{x}).
        \end{eqnarray*}
        Now, in the reference to Proposition \ref{prop24} and Note \ref{note3.3} for $H(g),H(\Phi)\in \mathcal{C(\mathfrak{F})}$ gives the required inequality.
    \end{enumerate}
\end{proof}
For $\xi\in [2\eta,2\eta+1]$, consider $\Lambda_\xi=\{k\Phi_\xi:~k\in\mathbb{R}\}\subseteq \mathcal{C}(\mathfrak{F}\times \mathfrak{F})$. Define a map $\Omega_\xi:\Lambda_\xi\to \mathbb{R}$ as
$$\Omega_\xi(k\Phi_\xi)=k,$$
and a measure $\mathscr{L}_{\Lambda_\xi}$ on $\Lambda_\xi$ by
$$\mathscr{L}_{\Lambda_\xi}=\mathscr{L}^1\circ \Omega_\xi$$
\begin{proposition}\label{prop29}
    $\Lambda_\xi$ is a probe space for $\mathfrak{A}_\xi(\mathfrak{F}\times \mathfrak{F})$. That is,
    $$\mathscr{L}_{\Lambda_\xi}\bigg(\mathcal{D}_\xi(\mathfrak{F}\times \mathfrak{F})\backslash\big(g+\mathfrak{A}_\xi(\mathfrak{F}\times \mathfrak{F})\big)\bigg)=0,$$
    for all $g\in \mathcal{D}_\xi(\mathfrak{F}\times \mathfrak{F})$.
\end{proposition}
\begin{proof}
    Let $g\in \mathcal{C(SG\times SG)}$. Since $\mathscr{L}_{\Lambda_\xi}$ is a measure on $\Lambda_\xi$ and by Proposition \ref{prop28}, we get
    \begin{eqnarray*}
        \mathscr{L}_{\Lambda_\xi}\big(\mathcal{D}_\xi(\mathfrak{F}\times \mathfrak{F})\backslash (g+\mathfrak{A}_\xi(\mathfrak{F}\times \mathfrak{F}))\big) &=& \mathscr{L}_{\Lambda_\xi}\big(\Lambda_\xi\backslash (g+\mathfrak{A}_\xi(\mathfrak{F}\times \mathfrak{F}))\big)\\
        &=& (\mathscr{L}^1\circ \Omega_\xi)\big(k\Phi_\xi\in \Lambda_\xi:~k\Phi_\xi-g\notin\mathfrak{A}_\xi(\mathfrak{F}\times \mathfrak{F})\big)\\
        &=& \mathscr{L}^1\big(k\in \mathbb{R}:~k\Phi_\xi-g\notin\mathfrak{A}_\xi(\mathfrak{F}\times \mathfrak{F})\big)\\
        &=& 0,
    \end{eqnarray*}
    establishes the claim.
\end{proof}

\noindent\textbf{Proof of Theorem \ref{thm12}.}
On combining Proposition \ref{prop26} and Proposition \ref{prop29} we get our Theorem \ref{thm12}.

Now, let us define a subspace $\mathcal{HC}^\sigma(\mathfrak{F}\times \mathfrak{F})$, which is a collection of H\"{o}lder continuous functions with exponent $\sigma~(=1+2\eta-\xi)\in (0,1]$, of $\mathcal{D}_\xi(\mathfrak{F}\times \mathfrak{F})$ as:
$$\mathcal{HC}^\sigma(\mathfrak{F}\times \mathfrak{F}):=\Bigg\{g\in \mathcal{D}_\xi(\mathfrak{F}\times \mathfrak{F})~\Bigg{|}~[g]_\sigma=\sup_{\substack{\mathfrak{Z},\hat{\mathfrak{Z}}\in \mathfrak{F}\times \mathfrak{F}\\ \mathfrak{Z}\ne \hat{\mathfrak{Z}}}}\frac{|g(\mathfrak{Z})-g(\hat{\mathfrak{Z}})|}{~\|\mathfrak{Z}-\hat{\mathfrak{Z}}\|_2^{\sigma}}<\infty \Bigg\},$$
One can easily verify that $\mathcal{HC}^\sigma(\mathfrak{F}\times \mathfrak{F})$ is complete w.r.t. the norm $\|g\|_{\mathcal{HC}^\sigma}=\|g\|_\infty+[g]_\sigma$. The corresponding metric is symbolized by $\Theta_{\mathcal{HC}^\sigma}$. Let us now see the horizon property for $\mathcal{HC}^\sigma$.
\begin{theorem}
    The set
    $$\Bigg\{g\in \mathcal{HC}^\sigma(\mathfrak{F}\times \mathfrak{F})~\Big{|}~\dim_B\big(\mathfrak{Gr}(g)\big)=1+2\eta-\sigma\text{ and }\dim_B\big(\mathfrak{Gr}(H(g))\big)=1+\eta-\sigma\Bigg\}$$
    is a prevalent subset of $\big(\mathcal{HC}^\sigma(\mathfrak{F}\times \mathfrak{F}),\Theta_{\mathcal{HC}^\sigma}\big)$.
\end{theorem}
\begin{proof}
    Our primary target is to show first that for $g\in \mathcal{HC}^\sigma(\mathfrak{F}\times \mathfrak{F})$, $H(g)\in \mathcal{HC}^\sigma(\mathfrak{F}).$ Thus, using the proof of Item (\ref{lem23_it2}), consider
    \begin{eqnarray*}
        [H(g)]_\sigma&=&\sup_{\substack{\mathbf{x}_1,\mathbf{x}_2\in \mathfrak{F}\\ \mathbf{x}_1\ne \mathbf{x}_2}}\frac{|H(g)(\mathbf{x}_1)-H(g)(\mathbf{x}_2)|}{~\|\mathbf{x}_1-\mathbf{x}_2\|_2^{\sigma}}\\
        &=&\sup_{\substack{\mathbf{x}_1,\mathbf{x}_2\in \mathfrak{F}\\ \mathbf{x}_1\ne \mathbf{x}_2}}~\sup_{\mathbf{y}\in \mathfrak{F}}\frac{|g(\mathbf{x}_1,\mathbf{y})-g(\mathbf{x}_2,\mathbf{y})|}{~\|\mathbf{x}_1-\mathbf{x}_2\|_2^{\sigma}}\\
        &=&\sup_{\substack{\mathbf{x}_1,\mathbf{x}_2\in \mathfrak{F}\\ \mathbf{x}_1\ne \mathbf{x}_2}}~\sup_{\mathbf{y}\in \mathfrak{F}}\frac{|g(\mathbf{x}_1,\mathbf{y})-g(\mathbf{x}_2,\mathbf{y})|}{\|(\mathbf{x}_1,\mathbf{y})-(\mathbf{x}_2,\mathbf{y})\|_2^{\sigma}}\\
        &<& \infty.
    \end{eqnarray*}
Since $g\in \mathcal{HC}^\sigma(\mathfrak{F}\times \mathfrak{F})$, then $\overline{\dim}_B\big(\mathfrak{Gr}(g)\big)\le 1+2\eta-\sigma$. Similarly, $H(g)\in \mathcal{HC}^\sigma(\mathfrak{F})$ implies that $\overline{\dim}_B\big(\mathfrak{Gr}(H(g))\big)\le 1+\eta-\sigma.$ On combining this with Proposition \ref{prop28}, we get
$$\dim_B\big(\mathfrak{Gr}(H(g+k\Phi_\xi))\big)=1+\eta-\sigma.$$
The rest proof is similar to the proof of Theorem \ref{thm12}.
\end{proof}

\begin{note}
    Here, we elucidate the fact that one cannot drop the OSC for the self-similar IFS of $\mathfrak{F}$ that we have assumed in Assumption \ref{assum2.1} as it plays a crucial role in establishing Proposition \ref{prop3.4}. To justify this, let $\mathfrak{I}=\{[0,1];~L_1,L_2\}$, where
    $$L_1(x)=\frac{3}{5} x \quad \text{and} \quad L_2(x)=\frac{3}{5} x+\frac{2}{5}.$$
    We seek to show that the IFS $\mathfrak{I}$ does not satisfy the OSC. For this, if possible, let there exist an open subset $\mathcal{U}(\ne \emptyset)$ of $\mathbb{R}$ such that $\mathfrak{I}$ follows OSC and let $\lambda(\mathcal{U})$ be the Lebesgue measure of $\mathcal{U}$. Thus, by the definition of OSC, we get
    \begin{eqnarray*}
        \lambda(\mathcal{U})&\ge& \frac{3}{5}\lambda(\mathcal{U})+\frac{3}{5}\lambda(\mathcal{U})\\
        &=& \frac{6}{5}\lambda(\mathcal{U}),
    \end{eqnarray*}
    which contradicts the fact that $\lambda(\mathcal{U})>0$. Therefore, the OSC fails for $\mathfrak{I}$. Clearly, the ranges of $L_1$ and $L_2$ are intersecting, and the overlapping region is $\mathcal{O}=\left[\frac{2}{5},\frac{3}{5}\right]$. Now, consider $\Sigma=\{1,2\}$, the corresponding IFS of $[0,1]\times[0,1]$ is $\widehat{\mathfrak{I}}=\{[0,1]\times[0,1];~(L_i(x),L_j(y)), i,j\in \Sigma\}$ and the associated overlapping region is a square $\mathcal{O\times O}$. Let us define a function $\psi:[0,1]\to \mathbb{R}$ as
    \begin{eqnarray*}
        \psi(x)=\begin{cases}
\exp\left(1-\frac{1}{1-\left(10(x-\frac{1}{2})\right)^2}\right), & \text{for } x \in \left(\frac{2}{5},\frac{3}{5}\right) \\
0, & \text{for }  x \notin \left(\frac{2}{5},\frac{3}{5}\right).
\end{cases}
    \end{eqnarray*}
    Subsequently, define $g:[0,1]\times[0,1]\to \mathbb{R}$ as
    $$g(x,y)=\psi(x)\psi(y).$$
    Then, $\text{supp}(g)=\mathcal{O\times O}$ and $g$ is infinitely times continuously differentiable on $[0,1]\times[0,1]$; consequently, for $\xi\in[2,3]$,
    $$\dim_B\big(\mathfrak{Gr}(g)\big)=2\implies g\in \mathcal{D}_\xi.$$
    One can observe that the function range of $g$ is (0,1] on $\left(\frac{2}{5},\frac{3}{5}\right)$. Consider $\mathcal{P}$ is the intersection of the product cell $\left(L_\omega(\mathfrak{F})\times L_\tau (\mathfrak{F})\right)$ and overlapping square $\left(\mathcal{O\times O}\right)$. Then, for any non-empty $\mathcal{P}$, there exists $c_0\in(0,1]$ such that 
    $$\mathcal{R}_g[L_\omega(\mathfrak{F})\times L_\tau (\mathfrak{F})]\big{|}_\mathcal{P}\ge c_0>0.$$ Thus,
    \begin{eqnarray*}
        Osc(m,g)&=& \sum_{\substack{\omega,\tau\in\Sigma^m \\ \mathcal{P}=\emptyset}} \mathcal{R}_{g}[L_\omega(\mathfrak{F})\times L_\tau (\mathfrak{F})]+\sum_{\substack{\omega,\tau\in\Sigma^m \\\mathcal{P}\ne\emptyset}} \mathcal{R}_g[L_\omega(\mathfrak{F})\times L_\tau (\mathfrak{F})]\\
        &\ge& 0+c_0 \sum_{\substack{\omega,\tau\in\Sigma^m \\\mathcal{P}\ne\emptyset}} 1\\
        &=& c_0 \beta 4^m
    \end{eqnarray*}
    where $\beta 4^m$ symbolizes the number of intersecting cells in proportion to the total $4^m$ cells. Since $|a|=\frac{3}{5}$ and $\xi\in [2,3]$, then
    \begin{eqnarray*}
        \frac{Osc(m,g_n-g)}{|a|^{-m\left(\xi-1\right)}}&\ge& \frac{c_0\beta 4^m}{(\frac{3}{5})^{-m(\xi-1)}}\\
        &=&c_0\beta\left(4\left(\frac{3}{5}\right)^{(\xi-1)}\right)^m\\
        &\ge&c_0\beta\left(\frac{36}{25}\right)^m
    \end{eqnarray*}
    Thus, as $m\to \infty$,
    $$\frac{Osc(m,g_n-g)}{|a|^{-m\left(\xi-1\right)}}\to \infty.$$
    Hence, Proposition \ref{prop3.4} fails, the OSC is the necessary one that cannot be dropped.
\end{note}

\section{Canonical Cases}\label{sec:4}
In this section, we discuss particular cases of the developed analytical results of the previous section. In the next two examples, we consider two notorious self-similar fractals from $\mathbb{R}$ and $\mathbb{R}^2$ and discuss how the results of Section \ref{sec:3} update for these particular cases.

\begin{example}\label{ex4.1}
    Consider the IFS $\mathfrak{J}_{[0,1]}=\{\mathbb{R};~L_1,L_2,L_3\}$ with 
    $$L_1(x)=\frac{1}{3} x,\quad L_2(x)=\frac{1}{3} x+\frac{1}{3},\quad L_3(x)=\frac{1}{3} x+\frac{2}{3}.$$
    Corresponding to this IFS, one can get its fractal $\mathfrak{F}=[0,1]\in\mathscr{K}(\mathbb{R})$ such that its associated IFS satisfies the OSC. Thus, by Note \ref{note2.2}, the box dimension of $[0,1]$ exists and $\eta=\dim_B\left([0,1]\right)=1$. Next, $|a|=\max_{1\le i \le 3}  \frac{1}{3}= \frac{1}{3}$, by Lemma \ref{lem3.1}, we can cover any continuous function $g$ on $[0,1]\times[0,1]$ as
    \begin{equation}
        3^m\sum_{\omega,\tau\in \Sigma^m}\mathcal{R}_g[L_\omega([0,1])\times L_\tau ([0,1])]\le N_\delta\big(\mathfrak{Gr}(g)\big)\le 2.3^m.3^m+3^m\sum_{\omega,\tau\in \Sigma^m}\mathcal{R}_g[L_\omega([0,1])\times L_\tau ([0,1])].\label{eqn4.1}
    \end{equation}
    The oscillation space $\mathbb{O}^{\xi}$ on $[0,1] \times [0,1],\text{ for}~\xi \in [2,3]$, is defined as
    $$ \mathbb{O}^{\xi}([0,1] \times [0,1])= \Big\{g \in \mathcal{C}([0,1] \times [0,1]): \sup_{m \in \mathbb{N} }\dfrac{ Osc(m,g)}{3^{m\left(\xi-1\right)}}< \infty \Big\}.$$
    The main result Theorem \ref{thm12} for this case is 
    \begin{enumerate}
        \item for each $\xi \in [2,3)$, the subset $\mathfrak{A}_{\xi}([0,1] \times [0,1])$ is a prevalent of $\big(\mathcal{D}_{\xi}([0,1] \times [0,1]),\Theta_\xi\big);$
        \item the set $\mathfrak{A}_{3}([0,1] \times [0,1])$ is a prevalent subset of $\big(\mathcal{C}([0,1] \times [0,1]),\Theta_\infty\big)$.
    \end{enumerate}
\end{example}
\begin{note}
    In reference to Remark \ref{rmk3.2}, if we consider oscillation space as the collection of all those continuous functions $g$ on $[0,1]\times [0,1]$ with $\|g\|_{\xi}<\infty$ , then Example \ref{ex4.1} will match with \cite{FF}. Hence, our result can be seen as a generalized version.
\end{note}

\begin{example}\label{ex4.2}
    For each $i\in\{1,2,3\}$, let $L_i:\mathbb{R}^2\to \mathbb{R}^2$ be defined as $$L_i (x,y)=\frac{1}{2}\left(x+q_1^{(i)},y+q_2^{(i)}\right),$$
    where $\left(q_1^{(i)},q_2^{(i)}\right)$ are the vertices of an equilateral triangle in $\mathbb{R}^2$. Subsequently, the IFS $\mathfrak{J}_{SG}=\left\{\mathbb{R}^2;~L_1,L_2,L_3\right\}$ yields the fractal as Sierpi\'nski gasket ($SG$), which lies in $\mathscr{K}(\mathbb{R}^2)$. Since IFS $\mathfrak{J}_{SG}$ also satisfies the OSC, in reference to Note \ref{note2.2}, its box dimension exists and $\eta=\dim_B\left(SG\right)=\frac{\log 3}{\log 2}$. For $|a|=\max_{1\le i \le 3}  \frac{1}{2}= \frac{1}{2}$, Lemma \ref{lem3.1} yields an inequality to cover any continuous function $g$ on $SG\times SG$ as 
    \begin{equation}
        2^m\sum_{\omega,\tau\in \Sigma^m}\mathcal{R}_g[L_\omega(SG)\times L_\tau (SG)]\le N_\delta\big(\mathfrak{Gr}(g)\big)\le 2.3^m.3^m+2^m\sum_{\omega,\tau\in \Sigma^m}\mathcal{R}_g[L_\omega(SG)\times L_\tau (SG)].\label{eqn4.2}
    \end{equation}
    For $\xi \in \left[\frac{2\log 3}{\log 2},\frac{2\log 3}{\log 2}+1\right]$, the corresponding oscillation space $\mathbb{O}^{\xi}$ on $SG \times SG$, is defined as
    $$ \mathbb{O}^{\xi}(SG \times SG)= \bigg\{g \in \mathcal{C}(SG \times SG): \sup_{m \in \mathbb{N} }\dfrac{ Osc(m,g)}{2^{m\left(\xi-1\right)}}< \infty \bigg\}.$$
    Theorem \ref{thm12} will be updated in this case as
    \begin{enumerate}
        \item for each $\xi \in \left[\frac{2\log 3}{\log 2},\frac{2\log 3}{\log 2}+1\right)$, the subset $\mathfrak{A}_{\xi}(SG \times SG)$ is a prevalent of $\big(\mathcal{D}_{\xi}(SG \times SG),\Theta_\xi\big);$
        \item the set $\mathfrak{A}_{\frac{2\log 3}{\log 2}+1}(SG \times SG)$ is a prevalent subset of $\big(\mathcal{C}(SG \times SG),\Theta_\infty\big)$.
    \end{enumerate}
\end{example}

\section{Conclusions}
We have discussed the dimensional results of the graphs of the continuous functions defined over $\mathfrak{F\times F}$, where $\mathfrak{F}$ is a self-similar fractal whose IFS satisfies the open set condition (OSC). Following this, we generalized the horizon problem of the prevalent surface over $\mathfrak{F\times F}$. Precisely, we demonstrated the existence of a prevalent surface, defined by the continuous function over $\mathfrak{F\times F}$, that satisfies the horizon property. Additionally, we emphasized the significance of OSC of the self-similar IFS of $\mathfrak{F}$ via an example justifying its indispensability. To consolidate this theory, we have included two illustrations covering the particular cases for the notorious fractals of $\mathbb{R}$ and $\mathbb{R}^2$.

\section*{Declaration}
\noindent
\textbf{Funding:}  This research did not receive any specific grant from funding agencies in the public, commercial, or not-for-profit sectors.\\
\textbf{Conflicts of interest.} The authors further confirm that there are no conflicts of interest associated with this work.\\
\noindent
\textbf{Data availability:} No datasets were employed in the conduct of this study.\\
\noindent
\textbf{Code availability:} Not applicable.\\
\textbf{Declaration of generative AI and AI-assisted technologies in the writing process:}
During the preparation of this work the author(s) did not make use of any generative AI.\\
\noindent
\textbf{Authors' contributions:} All authors have made an equal contribution to the preparation of this manuscript.
\section*{Acknowledgements}
The first author acknowledges the Prime Minister's Research Fellowship (PMRF) with PMRF ID: 2202749, Ministry of Education, Government of India, for the financial support during the Ph.D. work.

\bibliographystyle{elsarticle-num-names}
\bibliography{ref}

\end{document}